\documentclass[12pt]{article}
\usepackage{amsmath,amsthm,amssymb}
\usepackage{amssymb,latexsym}

\newtheorem{theorem}{Theorem}

\newtheorem{lemma}{Lemma}

\begin{document}
\baselineskip=17pt

\title{\bf Consecutive square-free values for some polynomials}

\author{\bf S. I. Dimitrov}

\date{2023}

\maketitle

\begin{abstract}

In this paper we study the distribution of consecutive square-free numbers of the forms
$x^2+y^2+z+1$, $x^2+y^2+z+2$ and $x^2+y^2+z^2+z+1$, $x^2+y^2+z^2+z+2$, respectively.
We establish asymptotic formulas for each of these two  cases.
{\bf  2020 Math.\ Subject Classification}:  11L05 $\cdot$ 11N25 $\cdot$  11N37
\end{abstract}

\section{Notations}
\indent

Let $H$ be a sufficiently large positive number.
By $\varepsilon$ we denote an arbitrary small positive number, not the same in all appearances.
As usual $\mu(n)$ is M\"{o}bius' function, $\omega(n)$ denotes the number of distinct prime factors of $n$
and $\tau(n)$ represents  the number of positive divisors of $n$.
Further $[t]$ and $\{t\}$ denote the integer part, respectively, the fractional part of $t$.
Instead of $m\equiv n\,\pmod {d}$ we write for simplicity $m\equiv n\,(d)$.
Moreover $(l,m)$ is the greatest common divisor of $l$ and $m$,
and $(l,m,n)$ is the greatest common divisor of $l$, $m$ and $n$.
The letter $p$  will always denote prime number.
We put $\|t\|=\min(\{t\}_,1-\{t\})$. As usual $e(t)$=exp($2\pi it$).
For any odd $q$ we denote by $\left(\frac{\cdot}{q}\right)$  the Jacobi symbol.
For any $n$ and $q$ such that $(n, q)=1$ we denote by $\overline{n}_q$
the inverse of $n$ modulo $q$.
Throughout this paper we shall denote the Gauss sums
\begin{equation}\label{Gausssums}
G(q,m,n)=\sum\limits_{x=1}^{q}e\left(\frac{mx^2+nx}{q}\right)\,,\quad G(q,m)=G(q,m,0)\,,
\end{equation}
the Sali\'{e} sum
\begin{equation}\label{Saliesum}
S(q,m,n)=\sum\limits_{x=1\atop{(x, q)=1}}^{q}\left(\frac{x}{q}\right)e\left(\frac{mx+n\bar{x}_q}{q}\right)
\end{equation}
and the Kloosterman sum
\begin{equation}\label{Kloosterman}
K(q,m,n)=\sum\limits_{x=1\atop{(x, q)=1}}^{q}e\left(\frac{mx+n\bar{x}_q}{q}\right)\,.
\end{equation}
Define
\begin{align}
\label{GammaH}
&\Gamma(H)=\sum\limits_{1\leq x, y,z\leq H}\mu^2(x^2+y^2+z^2+z+1)\mu^2(x^2+y^2+z^2+z+2)\,,\\
\label{GammaHast}
&\Gamma^\ast(H)=\sum\limits_{1\leq x, y,z\leq H}\mu^2(x^2+y^2+z+1)\mu^2(x^2+y^2+z+2)\,,\\
\label{lambdaq1q2lmn}
&\lambda(q_1, q_2, l, m, n)=\sum\limits_{1\leq x, y,z\leq q_1q_2\atop{x^2+y^2+z^2+z+1\equiv 0\,(q_1)
\atop{x^2+y^2+z^2+z+2\equiv 0\,(q_2)}}}e\left(\frac{lx+my+nz}{q_1q_2}\right)\,,\\
\label{lambdaq1q2lmnast}
&\lambda^\ast(q_1, q_2, l, m, n)=\sum\limits_{1\leq x, y,z\leq q_1q_2\atop{x^2+y^2+z+1\equiv 0\,(q_1)
\atop{x^2+y^2+z+2\equiv 0\,(q_2)}}}e\left(\frac{lx+my+nz}{q_1q_2}\right)\,.
\end{align}
We define also
\begin{equation}\label{lambdaq1q2}
\lambda(q_1, q_2 )=\lambda(q_1, q_2 , 0, 0, 0)
\end{equation}
and
\begin{equation}\label{lambdaq1q2ast}
\lambda^\ast(q_1, q_2 )=\lambda^\ast(q_1, q_2, 0, 0, 0)\,.
\end{equation}

\section{Introduction and statement of the results}
\indent

The problem for the consecutive square-free numbers arises in 1932 when Carlitz \cite{Carlitz} proved the asymptotic formula
\begin{equation}\label{Carlitz}
\sum\limits_{1\leq n\leq H}\mu^2(n)\mu^2(n+1)=\prod\limits_{p}\left(1-\frac{2}{p^2}\right)H
+\mathcal{O}\big(H^{\theta+\varepsilon}\big)\,,
\end{equation}
where $\theta=2/3$.
Afterwards formula \eqref{Carlitz} was improved by Heath-Brown \cite{Heath-Brown} to $\theta=7/11$
and by Reuss \cite{Reuss} to $\theta=(26+\sqrt{433})/81$.
An interesting problem for number theory is the proof of formula \eqref{Carlitz} with numbers $n$ of a special form.
In this regard in 2020 the author \cite{Dimitrov1} using the method of Tolev \cite{Tolev} showed that there exist
infinitely many consecutive square-free numbers of the form $x^2+y^2+1$, $x^2+y^2+2$.
More precisely we proved that the asymptotic formula
\begin{equation}\label{asymptoticformula0}
\sum\limits_{1\leq x, y\leq H}\mu^2(x^2+y^2+1)\mu^2(x^2+y^2+2)=
\prod\limits_{p}\left(1-\frac{\widetilde{\lambda}(p^2, 1)+\widetilde{\lambda}(1, p^2)}{p^4}\right)
H^2+\mathcal{O}\left(H^{\frac{8}{5}+\varepsilon}\right)
\end{equation}
holds. Here
\begin{equation*}
\widetilde{\lambda}(q_1, q_2)=\sum\limits_{1\leq x, y\leq q_1 q_2\atop{x^2+y^2+1\equiv 0\,(q_1)\atop{x^2+y^2+2\equiv 0\,(q_2)}}}1\,.
\end{equation*}
Subsequently the author \cite{Dimitrov2} proved that there exist infinitely many square-free pairs of the type $x^2+1$, $x^2+2$.
Further, Jing and Liu \cite{Jing} improved the  reminder term in \eqref{asymptoticformula0}
to $\mathcal{O}\left(H^{\frac{3}{2}+\varepsilon}\right)$.
Continuing these research  Zhou and Ding \cite{Ding} established that
\begin{equation*}
\sum\limits_{1\leq x, y,z\leq H}\mu^2(x^2+y^2+z^2+k)=\prod\limits_{p}\left(1-\frac{\lambda_3(p^2)}{p^6}\right) H^3+\mathcal{O}\left(H^{\frac{7}{3}+\varepsilon}\right)\,,
\end{equation*}
where
\begin{equation*}
\lambda_3(q)=\sum\limits_{1\leq x, y,z\leq q\atop{x^2+y^2+z^2+k\equiv 0\,(q)}}1\,.
\end{equation*}
Afterwards B. Chen \cite{Bo} generalized the results of the author by deriving the asymptotic formula
\begin{equation*}
\sum\limits_{1\leq x_1,\ldots, x_k\leq H}\mu^2(x_1^2+\cdots+x_k^2+1)\,\mu^2(x_1^2+\cdots+x_k^2+2)
=\sigma_k H^k+\mathcal{O}\left(H^{k-\frac{1}{2}-\frac{1}{2k}+\varepsilon}\right)\,,
\end{equation*}
where $k\geq3$,
\begin{equation*}
\sigma_k=\prod\limits_{p}\left(1-\frac{\lambda_k(p^2, 1)+\lambda_k(1, p^2)}{p^{2k}}\right)
\end{equation*}
and
\begin{equation*}
\lambda_k(q_1, q_2)=
\sum\limits_{1\leq  x_1,\ldots, x_k\leq q_1 q_2\atop{x_1^2+\cdots+x_k^2+1\equiv 0\,(q_1)\atop{x_1^2+\cdots+x_k^2+2\equiv 0\,(q_2)}}}1\,.
\end{equation*}
Recently Chen and Wang \cite{Chen} generalized the result of Zhou and Ding
by deriving an asymptotic formula for the distribution of $r$-free numbers of the form $x^2+y^2+z^2+k$.
In the case $r=2$, they obtained reminder term $\mathcal{O}\left(H^{\frac{9}{4}+\varepsilon}\right)$
which improves the reminder term $\mathcal{O}\left(H^{\frac{7}{3}+\varepsilon}\right)$ given by Zhou and Ding  \cite{Ding}.
Very recently Fan and Zhai \cite{Fan} established an asymptotic formula for the number of triples of positive integers
$x, y, z \leq H$ such that $x^2+y^2+z^2+1$, $x^2+y^2+z^2+2$ are $k$-free with $k\geq 2$.
Especially, in the case $k=2$, they obtained the error term $\mathcal{O}\left(H^{\frac{9}{4}+\varepsilon}\right)$
which improves the error term $\mathcal{O}\left(H^{\frac{7}{3}+\varepsilon}\right)$ given by B. Chen \cite{Bo}.
Motivated by these results we extend these studies by proving two theorems.
\begin{theorem}\label{Theorem1} For the sum $\Gamma(H)$ defined by \eqref{GammaH} we have the asymptotic formula
\begin{equation}\label{asymptoticformula1}
\Gamma(H)=\prod\limits_{p}\left(1-\frac{\lambda(p^2, 1)+\lambda(1, p^2)}{p^6}\right)H^3
+\mathcal{O}\left(H^{\frac{9}{4}+\varepsilon}\right)\,.
\end{equation}
\end{theorem}

\begin{theorem}\label{Theorem2} For the sum $\Gamma^\ast(H)$ defined by \eqref{GammaHast} we have the asymptotic formula
\begin{equation}\label{asymptoticformula2}
\Gamma^\ast(H)=\prod\limits_{p}\left(1-\frac{\lambda^\ast(p^2, 1)+\lambda^\ast(1, p^2)}{p^6}\right)H^3
+\mathcal{O}\left(H^{\frac{5}{2}+\varepsilon}\right)\,.
\end{equation}
\end{theorem}

\section{Lemmas}
\indent

\begin{lemma}\label{Gausslemma}
For the Gauss sum we have

\bigskip

\emph{(i)} \;\;\quad If $(q_1,q_2)=1$ then
\begin{equation*}
G(q_1 q_2, m, n)=G(q_1, mq_2, n)\,G(q_2, mq_1, n)\,.
\end{equation*}

\emph{(ii)} \quad If $(q_1,q_2)=1$ then
\begin{equation*}
G(q_1 q_2, m_1 q_2+m_2 q_1, n)=G(q_1, m_1 q_2^2, n)\,G(q_2, m_2 q_1^2, n)\,.
\end{equation*}
\quad \emph{(iii)} \quad If $(q,m)=d$ then
\begin{equation*}
G(q, m ,n)=\begin{cases}d\,G\left(q/d, m/d, n/d\right) \;\text{ if }\; d\mid{n}\,,\\
0 \quad\quad\quad\quad\quad\quad\quad\quad \mbox{ if } \; d\nmid{n}\,.
\end{cases}
\end{equation*}
\;\;\;\emph{(iv)} \quad If $(q ,2m)=1$ then
\begin{equation*}
G(q, m, n)=e\left(\frac{-\overline{(4m)}_q\,n^2}{q}\right)\left(\frac{m}{q}\right)G(q,1)\,.
\end{equation*}
\;\, \emph{(v)} \quad If $(q, 2)=1$ then
\begin{equation*}
G^2(q, 1)=(-1)^{\frac{q-1}{2}} q\,.
\end{equation*}
\end{lemma}
\begin{proof}
See \cite{Estermann} and \cite{Hua}.
\end{proof}

\begin{lemma}\label{Salieestimate}
Let $q$ is odd integer. For the sum defined by \eqref{Saliesum} we have
\begin{equation*}
|S(q,m,n)|\leq 2^{\omega(q)}\sqrt{q}\,.
\end{equation*}
\end{lemma}
\begin{proof}
See \cite{Louvel}.
\end{proof}

\begin{lemma}\label{Weilsestimate}
For the sum defined by \eqref{Kloosterman} we have
\begin{equation*}
|K(q,m,n)|\leq \tau(q)\,q^{\frac{1}{2}}\,(q,m,n)^{\frac{1}{2}}\,.
\end{equation*}
\end{lemma}
\begin{proof}
See \cite{Iwaniec}.
\end{proof}

\begin{lemma}\label{minimum}
For any real number $\xi$ and all integers $N_1$, $N_2$ with $N_1< N_2$, we have
\begin{equation*}
\sum\limits_{n=N_1+1}^{N_2} e(\xi n)\ll\min \left\{N_2-N_1,\,  \|\xi\|^{-1} \right\}\,.
\end{equation*}
\end{lemma}
\begin{proof}
See (\cite{Nathanson}, Lemma 4.7).
\end{proof}

\begin{lemma}\label{multiplicative1}
Let
\begin{equation}\label{coprimeq1234}
(q_1 q_2, q_3 q_4)=(q_1, q_2)=(q_3, q_4)=1\,.
\end{equation}
For the function defined by \eqref{lambdaq1q2lmn} we have
\begin{align*}
\lambda(q_1 q_2, q_3 q_4, l, m, n)
&=\lambda\left(q_1, q_3, l\overline{(q_2 q_4)}_{q_1 q_3},
m\overline{(q_2 q_4)}_{q_1 q_3}, n\overline{(q_2 q_4)}_{q_1 q_3}\right)\\
&\times\lambda\left(q_2, q_4, l\overline{(q_1 q_3)}_{q_2 q_4},
m\overline{(q_1 q_3)}_{q_2 q_4}, n\overline{(q_1 q_3)}_{q_2 q_4}\right)\,.
\end{align*}
\end{lemma}
\begin{proof}
On the one hand \eqref{Gausssums}, \eqref{lambdaq1q2lmn}, \eqref{coprimeq1234}, the well-known formula
\begin{equation}\label{well-known}
\sum\limits_{k=1}^{d}e\left(\frac{kn}{d}\right)
=\begin{cases}d\,,\;\mbox{  if  } \; d\,|\,n,\\
0\,,\; \mbox{  if } \; d\nmid n
\end{cases}
\end{equation}
and Lemma \ref{Gausslemma} give us
\begin{align*}
&\lambda(q_1q_2, q_3q_4, l, m, n)
=\frac{1}{q_1q_2q_3q_4}\sum\limits_{1\leq x, y,z\leq q_1q_2q_3q_4}e\left(\frac{lx+my+nz}{q_1q_2q_3q_4}\right)\nonumber\\
&\times\sum\limits_{1\leq h_1\leq q_1q_2}e\left(\frac{h_1(x^2+y^2+z^2+z+1)}{q_1q_2}\right)
\sum\limits_{1\leq h_2\leq q_3q_4}e\left(\frac{h_2(x^2+y^2+z^2+z+2)}{q_3q_4}\right)\nonumber\\
&=\frac{1}{q_1q_2q_3q_4}\sum\limits_{1\leq h_1\leq q_1q_2}e\left(\frac{h_1}{q_1q_2}\right)
\sum\limits_{1\leq h_2\leq q_3q_4}e\left(\frac{2h_2}{q_3q_4}\right)G(q_1q_2q_3q_4, h_1q_3q_4+h_2q_1q_2, l)\nonumber\\
&\times G(q_1q_2q_3q_4, h_1q_3q_4+h_2q_1q_2, m)\,G(q_1q_2q_3q_4, h_1q_3q_4+h_2q_1q_2, n+h_1q_3q_4+h_2q_1q_2)\nonumber\\
&=\frac{1}{q_1q_2q_3q_4}\sum\limits_{1\leq h_1\leq q_1q_2}e\left(\frac{h_1}{q_1q_2}\right)
G(q_1q_2, h_1q^2_3q^2_4, l)G(q_1q_2, h_1q^2_3q^2_4, m)G(q_1q_2, h_1q^2_3q^2_4, n+h_1q_3q_4)\nonumber\\
&\times\sum\limits_{1\leq h_2\leq q_3q_4}e\left(\frac{2h_2}{q_3q_4}\right)\,
G(q_3q_4, h_2q^2_1q^2_2, l)\,G(q_3q_4, h_2q^2_1q^2_2, m)\,G(q_3q_4, h_2q^2_1q^2_2, n+h_2q_1q_2)\nonumber\\
\end{align*}

\begin{align}\label{lambdaq1234lmnest1}
&=\frac{1}{q_1q_2q_3q_4}\sum\limits_{1\leq h_1\leq q_1\atop{1\leq h_2\leq q_2}}e\left(\frac{h_1q_2+h_2q_1}{q_1q_2}\right)
G(q_1q_2, h_1q_2q^2_3q^2_4+h_2q_1q^2_3q^2_4, l)\nonumber\\
&\times G(q_1q_2, h_1q_2q^2_3q^2_4+h_2q_1q^2_3q^2_4, m)\,
G(q_1q_2, h_1q_2q^2_3q^2_4+h_2q_1q^2_3q^2_4, n+h_1q_2q_3q_4+h_2q_1q_3q_4)\nonumber\\
&\times\sum\limits_{1\leq h_3\leq q_3\atop{1\leq h_4\leq q_4}}e\left(\frac{2(h_3q_4+h_4q_3)}{q_3q_4}\right)
G(q_3q_4, h_3q^2_1q^2_2q_4+h_4q^2_1q^2_2q_3, l)\nonumber\\
&\times G(q_3q_4, h_3q^2_1q^2_2q_4+h_4q^2_1q^2_2q_3, m)\,
G(q_3q_4, h_3q^2_1q^2_2q_4+h_4q^2_1q^2_2q_3, n+h_3q_1q_2q_4+h_4q_1q_2q_3)\nonumber\\
&=\frac{1}{q_1q_2q_3q_4}\sum\limits_{1\leq h_1\leq q_1\atop{1\leq h_2\leq q_2}}e\left(\frac{h_1q_2+h_2q_1}{q_1q_2}\right)
G(q_1, h_1q^2_2q^2_3q^2_4, l)\,G(q_2, h_2q^2_1q^2_3q^2_4, l)\,G(q_1, h_1q^2_2q^2_3q^2_4, m)\nonumber\\
&\times G(q_2, h_2q^2_1q^2_3q^2_4, m)\,G(q_1, h_1q^2_2q^2_3q^2_4, n+h_1q_2q_3q_4)\,
G(q_2, h_2q^2_1q^2_3q^2_4, n+h_2q_1q_3q_4)\nonumber\\
&\times\sum\limits_{1\leq h_3\leq q_3\atop{1\leq h_4\leq q_4}}e\left(\frac{2(h_3q_4+h_4q_3)}{q_3q_4}\right)
G(q_3, h_3q^2_1q^2_2q^2_4, l)\,G(q_4, h_4q^2_1q^2_2q^2_3, l)\,G(q_3, h_3q^2_1q^2_2q^2_4, m)\nonumber\\
&\times G(q_4, h_4q^2_1q^2_2q^2_3, m)\,G(q_3, h_3q^2_1q^2_2q^2_4, n+h_3q_1q_2q_4)\, G(q_4, h_4q^2_1q^2_2q^2_3, n+h_4q_1q_2q_3)\,.
\end{align}
On the other hand \eqref{Gausssums}, \eqref{lambdaq1q2lmn},  \eqref{coprimeq1234}, \eqref{well-known} and Lemma \ref{Gausslemma} imply
\begin{align*}
&\lambda\left(q_1, q_3, l\overline{(q_2q_4)}_{q_1q_3}, m\overline{(q_2q_4)}_{q_1q_3}, n\overline{(q_2q_4)}_{q_1q_3}\right)
\lambda\left(q_2, q_4, l\overline{(q_1 q_3)}_{q_2 q_4}, m\overline{(q_1 q_3)}_{q_2 q_4}, n\overline{(q_1q_3)}_{q_2q_4}\right)\nonumber\\
&=\frac{1}{q_1q_2q_3q_4}\sum\limits_{1\leq x_1, y_1, z_1\leq q_1q_3}
e\left(\frac{l\overline{(q_2q_4)}_{q_1q_3}x_1+m\overline{(q_2q_4)}_{q_1q_3}y_1+ n\overline{(q_2q_4)}_{q_1q_3}z_1}{q_1q_3}\right)\nonumber\\
&\times \sum\limits_{1\leq h_1\leq q_1}e\left(\frac{h_1(x_1^2+y_1^2+z_1^2+z_1+1)}{q_1}\right)
\sum\limits_{1\leq h_3\leq q_3}e\left(\frac{h_3(x_1^2+y_1^2+z_1^2+z_1+2)}{q_3}\right)\nonumber\\
&\times\sum\limits_{1\leq x_2, y_2, z_2\leq q_2q_4}
e\left(\frac{l\overline{(q_1q_3)}_{q_2q_4}x_2+m\overline{(q_1q_3)}_{q_2q_4}y_2+ n\overline{(q_1q_3)}_{q_2q_4}z_2}{q_2q_4}\right)\nonumber\\
&\times \sum\limits_{1\leq h_2\leq q_2}e\left(\frac{h_2(x_2^2+y_2^2+z_2^2+z_2+1)}{q_2}\right)
\sum\limits_{1\leq h_4\leq q_4}e\left(\frac{h_4(x_2^2+y_2^2+z_2^2+z_2+2)}{q_4}\right)\nonumber\\
&=\frac{1}{q_1q_2q_3q_4}\sum\limits_{1\leq h_1\leq q_1\atop{1\leq h_2\leq q_2}}e\left(\frac{h_1q_2+h_2q_1}{q_1q_2}\right)
\sum\limits_{1\leq h_3\leq q_3\atop{1\leq h_4\leq q_4}}e\left(\frac{2(h_3q_4+h_4q_3)}{q_3q_4}\right)\nonumber\\
&\times G\big(q_1q_3,h_1q_3+h_3q_1, l\overline{(q_2q_4)}_{q_1q_3}\big)\,
G\big(q_1q_3,h_1q_3+h_3q_1, m\overline{(q_2q_4)}_{q_1q_3}\big) \nonumber\\
&\times G\big(q_1q_3,h_1q_3+h_3q_1, n\overline{(q_2q_4)}_{q_1q_3}+h_1q_3+h_3q_1\big)\,
G\big(q_2q_4,h_2q_4+h_4q_2, l\overline{(q_1q_3)}_{q_2q_4}\big)\nonumber\\
&\times G\big(q_2q_4,h_2q_4+h_4q_2, m\overline{(q_1q_3)}_{q_2q_4}\big)\,
G\big(q_2q_4,h_2q_4+h_4q_2, n\overline{(q_1q_3)}_{q_2q_4}+h_2q_4+h_4q_2\big)\nonumber\\
\end{align*}

\begin{align}\label{lambdaq1234lmnest2}
&=\frac{1}{q_1q_2q_3q_4}\sum\limits_{1\leq h_1\leq q_1\atop{1\leq h_2\leq q_2}}e\left(\frac{h_1q_2+h_2q_1}{q_1q_2}\right)
G\big(q_1, h_1q^2_3, l\overline{(q_2q_4)}_{q_1q_3}\big)\,G\big(q_2,h_2q^2_4, l\overline{(q_1q_3)}_{q_2q_4}\big)\nonumber\\
&\times G\big(q_1, h_1q^2_3, m\overline{(q_2q_4)}_{q_1q_3}\big)\,G\big(q_2,h_2q^2_4, m\overline{(q_1q_3)}_{q_2q_4}\big)\nonumber\\
&\times G\big(q_1, h_1q^2_3,  n\overline{(q_2q_4)}_{q_1q_3}+h_1q_3\big)\,
G\big(q_2, h_2q^2_4, n\overline{(q_1q_3)}_{q_2q_4}+h_2q_4\big)\nonumber\\
&\times\sum\limits_{1\leq h_3\leq q_3\atop{1\leq h_4\leq q_4}}e\left(\frac{2(h_3q_4+h_4q_3)}{q_3q_4}\right)
G\big(q_3, h_3q^2_1, l\overline{(q_2q_4)}_{q_1q_3}\big)\, G\big(q_4, h_4q^2_2, l\overline{(q_1q_3)}_{q_2q_4}\big)\nonumber\\
&\times G\big(q_3, h_3q^2_1, m\overline{(q_2q_4)}_{q_1q_3}\big)\,G\big(q_4, h_4q^2_2, m\overline{(q_1q_3)}_{q_2q_4}\big)\nonumber\\
&\times G\big(q_3, h_3q^2_1, n\overline{(q_2q_4)}_{q_1q_3}+h_3q_1\big)\,
G\big(q_4, h_4q^2_2, n\overline{(q_1q_3)}_{q_2q_4}+h_4q_2\big)\,.
\end{align}
Using the substitution $x\rightarrow \overline{(q_2q_4)}_{q_1q_3}x$ we obtain
\begin{align}\label{Gausssum1}
G(q_1, h_1q^2_2q^2_3q^2_4, l)&=\sum\limits_{x=1}^{q_1}e\left(\frac{ h_1q^2_2q^2_3q^2_4x^2+lx}{q_1}\right)
=\sum\limits_{x=1}^{q_1}e\left(\frac{h_1q^2_3 x^2+l\overline{(q_2q_4)}_{q_1q_3}x}{q_1}\right)\nonumber\\
&=G\big(q_1, h_1q^2_3, l\overline{(q_2q_4)}_{q_1q_3}\big)\,.
\end{align}
Proceeding in a similar way, we get
\begin{align}
\label{Gausssum2}
&G(q_2, h_2q^2_1q^2_3q^2_4, l)=G\big(q_2,h_2q^2_4, l\overline{(q_1q_3)}_{q_2q_4}\big)\,,\\
\label{Gausssum3}
&G(q_1, h_1q^2_2q^2_3q^2_4, m)=G\big(q_1, h_1q^2_3, m\overline{(q_2q_4)}_{q_1q_3}\big)\,,\\
\label{Gausssum4}
&G(q_2, h_2q^2_1q^2_3q^2_4, m)=G\big(q_2,h_2q^2_4, m\overline{(q_1q_3)}_{q_2q_4}\big)\,,\\
\label{Gausssum5}
&G(q_1, h_1q^2_2q^2_3q^2_4, n+h_1q_2q_3q_4)=G\big(q_1, h_1q^2_3,  n\overline{(q_2q_4)}_{q_1q_3}+h_1q_3\big)\,,\\
\label{Gausssum6}
&G(q_2, h_2q^2_1q^2_3q^2_4, n+h_2q_1q_3q_4)=G\big(q_2, h_2q^2_4, n\overline{(q_1q_3)}_{q_2q_4}+h_2q_4\big)\,,\\
\label{Gausssum7}
&G(q_3, h_3q^2_1q^2_2q^2_4, l)=G\big(q_3, h_3q^2_1, l\overline{(q_2q_4)}_{q_1q_3}\big)\,,\\
\label{Gausssum8}
&G(q_4, h_4q^2_1q^2_2q^2_3, l)=G\big(q_4, h_4q^2_2, l\overline{(q_1q_3)}_{q_2q_4}\big)\,,\\
\label{Gausssum9}
&G(q_3, h_3q^2_1q^2_2q^2_4, m)=G\big(q_3, h_3q^2_1, m\overline{(q_2q_4)}_{q_1q_3}\big)\,,\\
\label{Gausssum10}
&G(q_4, h_4q^2_1q^2_2q^2_3, m)=G\big(q_4, h_4q^2_2, m\overline{(q_1q_3)}_{q_2q_4}\big)\,,\\
\label{Gausssum11}
&G(q_3, h_3q^2_1q^2_2q^2_4, n+h_3q_1q_2q_4)=G\big(q_3, h_3q^2_1, n\overline{(q_2q_4)}_{q_1q_3}+h_3q_1\big)\,,\\
\label{Gausssum12}
&G(q_4, h_4q^2_1q^2_2q^2_3, n+h_4q_1q_2q_3)=G\big(q_4, h_4q^2_2, n\overline{(q_1q_3)}_{q_2q_4}+h_4q_2\big)\,.
\end{align}
Bearing in mind \eqref{lambdaq1234lmnest1} -- \eqref{Gausssum12} we complete the proof of the lemma.
\end{proof}

\begin{lemma}\label{lambdaupperbound}
Let $8\nmid q_1q_2$ and $(q_1, q_2)=1$.
For the function defined by \eqref{lambdaq1q2lmn} the upper bound
\begin{equation*}
\lambda(q_1, q_2, l, m, n)\ll q_1q_2\tau(q_1q_2)2^{\omega(q_1)}2^{\omega(q_2)}(q_1q_2, l,m,n)
\end{equation*}
holds.
In particular we have
\begin{equation*}
\lambda(q_1, q_2, l, m, n)\ll (q_1 q_2)^{1+\varepsilon}(q_1q_2, l,m,n)\quad\mbox{ and }
\quad\lambda(q_1, q_2)\ll(q_1 q_2)^{2+\varepsilon}\,.
\end{equation*}
\end{lemma}

\begin{proof}

\textbf{Case 1.} $2\nmid q_1q_2$.

From \eqref{Gausssums}, \eqref{Saliesum}, \eqref{lambdaq1q2lmn}, \eqref{well-known} and Lemma \ref{Gausslemma} we deduce

\begin{align*}
&\lambda(q_1, q_2, l, m, n)
=\frac{1}{q_1q_2}\sum\limits_{1\leq x, y,z\leq q_1q_2}e\left(\frac{lx+my+nz}{q_1q_2}\right)\nonumber\\
&\times\sum\limits_{1\leq h_1\leq q_1}e\left(\frac{h_1(x^2+y^2+z^2+z+1)}{q_1}\right)
\sum\limits_{1\leq h_2\leq q_2}e\left(\frac{h_2(x^2+y^2+z^2+z+1)}{q_2}\right)\nonumber\\
&=\frac{1}{q_1q_2}\sum\limits_{1\leq h_1\leq q_1}e\left(\frac{h_1}{q_1}\right)
\sum\limits_{1\leq h_2\leq q_2}e\left(\frac{2h_2}{q_2}\right) G(q_1 q_2, h_1 q_2+h_2 q_1, l)\nonumber\\
&\times G(q_1 q_2, h_1 q_2+h_2 q_1, m)\,G(q_1 q_2, h_1 q_2+h_2 q_1, n+h_1 q_2+h_2 q_1) \nonumber\\
&=\frac{1}{q_1 q_2}\sum\limits_{1\leq h_1\leq q_1}e\left(\frac{h_1}{q_1}\right)
G(q_1, h_1 q^2_2, l)\, G(q_1, h_1 q^2_2, m)\, G(q_1, h_1 q^2_2, n+h_1 q_2)\nonumber\\
&\times\sum\limits_{1\leq h_2\leq q_2}e\left(\frac{2h_2}{q_2}\right)
G(q_2, h_2 q^2_1, l)\, G(q_2, h_2 q^2_1, m)\, G(q_2, h_2 q^2_1, n+h_2 q_1)\nonumber\\
&=\frac{1}{q_1 q_2}\sum\limits_{d_1 | q_1}\sum\limits_{1\leq h_1\leq q_1\atop{(h_1, q_1)=\frac{q_1}{d_1}}}
e\left(\frac{h_1}{q_1}\right)G(q_1, h_1 q^2_2, l)\, G(q_1, h_1 q^2_2, m)\, G(q_1, h_1 q^2_2, n+h_1 q_2)\nonumber\\
&\times\sum\limits_{d_2 | q_2}\sum\limits_{1\leq h_2\leq q_2\atop{(h_2, q_2)=\frac{q_2}{d_2}}}
e\left(\frac{2h_2}{q_2}\right)G(q_2, h_2 q^2_1, l)\, G(q_2, h_2 q^2_1, m)\, G(q_2, h_2 q^2_1, n+h_2 q_1)\nonumber\\
&=q^2_1 q^2_2\sum\limits_{d_1 | q_1\atop{\frac{q_1}{d_1} | (l, m, n)}}\frac{1}{d^3_1}
\sum\limits_{1\leq r_1\leq d_1\atop{(r_1, d_1)=1}}e\left(\frac{r_1}{d_1}\right)
G(d_1, r_1 q^2_2, l d_1 q^{-1}_1) G(d_1, r_1 q^2_2, m d_1 q^{-1}_1) G(d_1, r_1 q^2_2, n d_1 q^{-1}_1+r_1q_2)\nonumber\\
&\times\sum\limits_{d_2 | q_2\atop{\frac{q_2}{d_2} | (l, m, n)}}\frac{1}{d^3_2}
\sum\limits_{1\leq r_2\leq d_2\atop{(r_2, d_2)=1}}e\left(\frac{2r_2}{d_2}\right)
G(d_2, r_2 q^2_1, l d_2 q^{-1}_2) G(d_2, r_2 q^2_1, m d_2 q^{-1}_2) G(d_2, r_2 q^2_1, n d_2 q^{-1}_2+r_2q_1)\nonumber\\
&=q^2_1 q^2_2\sum\limits_{d_1 | q_1\atop{\frac{q_1}{d_1} | (l, m, n)}}\frac{G^3(d_1, 1)}{d^3_1}
\sum\limits_{1\leq r_1\leq d_1\atop{(r_1, d_1)=1}}\left(\frac{r_1}{d_1}\right)^3\nonumber\\
&\times e\left(\frac{r_1-\overline{(4 r_1 q^2_2)}_{d_1}\big(r^2_1q^2_2+(l^2+m^2+n^2)d^2_1 q^{-2}_1+2r_1q_2nd_1q^{-1}_1\big)}{d_1}\right)\nonumber\\
\end{align*}

\begin{align}\label{lambdaq1q2lmnest1}
&\times\sum\limits_{d_2 | q_2\atop{\frac{q_2}{d_2} | (l, m, n)}}\frac{G^3(d_2, 1)}{d^3_2}
\sum\limits_{1\leq r_2\leq d_2\atop{(r_2, d_2)=1}}\left(\frac{r_2}{d_2}\right)^3\nonumber\\
&\times e\left(\frac{2r_2-\overline{(4 r_2 q^2_1)}_{d_2}\big(r^2_2q^2_1+(l^2+m^2+n^2)d^2_2 q^{-2}_2+2r_2q_1nd_2q^{-1}_2\big)}{d_2}\right)\nonumber\\
&=q^2_1 q^2_2\sum\limits_{d_1 | q_1\atop{\frac{q_1}{d_1} | (l, m, n)}}
e\left(-\frac{\overline{(2q_2)}_{d_1}n}{q_1}\right)\frac{G^3(d_1, 1)}{d^3_1}
S\big(d_1, 1-\overline{4}_{d_1}, -\overline{(4q^2_2)}_{d_1}(l^2+m^2+n^2)d^2_1 q^{-2}_1\big)\nonumber\\
&\times\sum\limits_{d_2 | q_2\atop{\frac{q_2}{d_2} | (l, m, n)}}
e\left(-\frac{\overline{(2q_1)}_{d_2}n}{q_2}\right)\frac{G^3(d_2, 1)}{d^3_2}
S\big(d_2, 2-\overline{4}_{d_2}, -\overline{(4q^2_1)}_{d_2}(l^2+m^2+n^2)d^2_2 q^{-2}_2\big)\,.
\end{align}
Now \eqref{lambdaq1q2lmnest1}, Lemma \ref{Gausslemma} and Lemma \ref{Salieestimate} yield
\begin{align}\label{lambdaq1q2lmnest2}
\lambda(q_1, q_2, l, m, n)&\ll q^2_1 q^2_2
\sum\limits_{d_1 | q_1\atop{\frac{q_1}{d_1} | (l, m, n)}}d^{-3}_1d^\frac{3}{2}_1d^\frac{1}{2}_12^{\omega(d_1)}
\sum\limits_{d_2 | q_2\atop{\frac{q_2}{d_2} | (l, m, n)}}d^{-3}_2d^\frac{3}{2}_2d^\frac{1}{2}_22^{\omega(d_2)}\nonumber\\
&\ll q^2_1 q^2_2 2^{\omega(q_1)} 2^{\omega(q_2)}\sum\limits_{d_1 | q_1\atop{\frac{q_1}{d_1} | (l, m, n)}}d^{-1}_1
\sum\limits_{d_2 | q_2\atop{\frac{q_2}{d_2} | (l, m, n)}}d^{-1}_2\nonumber\\
&\ll q^2_1 q^2_2 2^{\omega(q_1)} 2^{\omega(q_2)}\sum\limits_{r_1 | (q_1, l, m, n)}q^{-1}_1r_1
\sum\limits_{r_2 | (q_2, l, m, n)}q^{-1}_2r_2\nonumber\\
&\ll q_1q_2\tau(q_1q_2)2^{\omega(q_1)}2^{\omega(q_2)}(q_1q_2, l,m,n)\,.
\end{align}

\textbf{Case 2.} $q_1=2^h q'_1$, where $2\nmid q'_1$, $h\leq2$ and $2\nmid q_2$.

Using \eqref{lambdaq1q2lmnest2}, Lemma \ref{multiplicative1} and the
trivial estimate $|\lambda(2^h, 1, l, m, n)|\leq8^h$ we obtain
\begin{align}\label{lambdaq1q2lmnest3}
\lambda(2^h q'_1, q_2, l, m, n)&=
\lambda\left(2^h, 1, l\overline{(q'_1 q_2)}_{2^h}, m \overline{(q'_1 q_2)}_{2^h}, n \overline{(q'_1 q_2)}_{2^h}\right)\nonumber\\
&\times\lambda\left(q'_1, q_2, l \overline{(2^h)}_{q'_1 q_2}, m \overline{(2^h)}_{q'_1 q_2}, n \overline{(2^h)}_{q'_1 q_2}\right)\nonumber\\
&\ll q'_1 q_2\tau(q'_1 q_2)2^{\omega(q'_1)}2^{\omega(q_2)}
\left(q'_1 q_2, l \overline{(2^h)}_{q'_1 q_2}, m \overline{(2^h)}_{q'_1 q_2}, n \overline{(2^h)}_{q'_1 q_2}\right)\nonumber\\
&\ll q_1 q_2\tau(q_1 q_2)2^{\omega(q_1)}2^{\omega(q_2)} (q_1 q_2, l, m, n)\,.
\end{align}

\textbf{Case 3.} $q_2=2^h q'_2$, where $2\nmid q'_2$, $h\leq2$ and $2\nmid q_1$.

By \eqref{lambdaq1q2lmnest2}, Lemma \ref{multiplicative1} and the trivial estimate $|\lambda(1, 2^h, l, m, n)|\leq8^h$ we get
\begin{align}\label{lambdaq1q2lmnest4}
\lambda(q_1, 2^h q'_2, l, m, n)&=
\lambda\left(1, 2^h, l\overline{(q_1 q'_2)}_{2^h}, m \overline{(q_1 q'_2)}_{2^h}, n \overline{(q_1 q'_2)}_{2^h}\right)\nonumber\\
&\times\lambda\left(q_1, q'_2, l\overline{(2^h)}_{q_1 q'_2}, m\overline{(2^h)}_{q_1q'_2}, n\overline{(2^h)}_{q_1q'_2}\right)\nonumber\\
&\ll q_1 q'_2\tau(q_1 q'_2)2^{\omega(q_1)}2^{\omega(q'_2)}
\left(q_1 q'_2, l \overline{(2^h)}_{q_1 q'_2}, m \overline{(2^h)}_{q_1 q'_2}, n \overline{(2^h)}_{q_1 q'_2}\right)\nonumber\\
&\ll q_1 q_2\tau(q_1 q_2)2^{\omega(q_1)}2^{\omega(q_2)} (q_1 q_2, l, m, n)\,.
\end{align}
Now the lemma follows from \eqref{lambdaq1q2lmnest2} -- \eqref{lambdaq1q2lmnest4}.
\end{proof}
Using Lemma \ref{lambdaupperbound} and arguing as in \cite{Fan} we derive the following lemma.
\begin{lemma}\label{Lambda123est}
Assume that $8\nmid q_1q_2$ and $H_0\geq2$. Then for the sums
\begin{equation*}%\label{Lambda12}
\Lambda_1=\sum\limits_{1\leq l\leq H_0}\frac{|\lambda(q_1, q_2, l, 0, 0)|}{l}\,,\quad
\Lambda_2=\sum\limits_{1\leq n\leq H_0}\frac{|\lambda(q_1, q_2, 0, 0, n)|}{n}\,,
\end{equation*}
\begin{equation*}%\label{Lambda34}
\Lambda_3=\sum\limits_{1\leq l, m\leq H_0}\frac{|\lambda(q_1, q_2, l, m, 0)|}{l m}\,,\quad
\Lambda_4=\sum\limits_{1\leq l, n\leq H_0}\frac{|\lambda(q_1, q_2, l, 0, n)|}{l n}
\end{equation*}
and
\begin{equation*}%\label{Lambda5}
\Lambda_5=\sum\limits_{1\leq l, m, n\leq H_0}\frac{|\lambda(q_1, q_2, l, m, n)|}{l m n}
\end{equation*}
the estimations
\begin{equation*}%\label{Lambda12est}
\Lambda_i\ll (q_1q_2)^{1+\varepsilon} H^\varepsilon_0\,,\quad  i=1, 2, 3, 4, 5
\end{equation*}
hold.
\end{lemma}

\begin{lemma}\label{multiplicative2}
Let
\begin{equation}\label{2coprimeq1234}
(q_1 q_2, q_3 q_4)=(q_1, q_2)=(q_3, q_4)=1\,.
\end{equation}
For the function defined by \eqref{lambdaq1q2lmnast} we have
\begin{align*}
\lambda^\ast(q_1 q_2, q_3 q_4, l, m, n)
&=\lambda^\ast\left(q_1, q_3, l\overline{(q_2 q_4)}_{q_1 q_3},
m\overline{(q_2 q_4)}_{q_1 q_3}, n\overline{(q_2 q_4)}_{q_1 q_3}\right)\\
&\times\lambda^\ast\left(q_2, q_4, l\overline{(q_1 q_3)}_{q_2 q_4},
m\overline{(q_1 q_3)}_{q_2 q_4}, n\overline{(q_1 q_3)}_{q_2 q_4}\right)\,.
\end{align*}
\end{lemma}
\begin{proof}
On the one hand \eqref{Gausssums}, \eqref{lambdaq1q2lmnast}, \eqref{well-known}, \eqref{2coprimeq1234} and Lemma \ref{Gausslemma} yield
\begin{align*}
&\lambda^\ast(q_1q_2, q_3q_4, l, m, n)
=\frac{1}{q_1q_2q_3q_4}\sum\limits_{1\leq x, y,z\leq q_1q_2q_3q_4}e\left(\frac{lx+my+nz}{q_1q_2q_3q_4}\right)\nonumber\\
&\times\sum\limits_{1\leq h_1\leq q_1q_2}e\left(\frac{h_1(x^2+y^2+z+1)}{q_1q_2}\right)
\sum\limits_{1\leq h_2\leq q_3q_4}e\left(\frac{h_2(x^2+y^2+z+2)}{q_3q_4}\right)\nonumber\\
\end{align*}

\begin{align}\label{lambdaastq1234lmnest1}
&=\frac{1}{q_1q_2q_3q_4}\sum\limits_{1\leq h_1\leq q_1q_2}e\left(\frac{h_1}{q_1q_2}\right)
\sum\limits_{1\leq h_2\leq q_3q_4}e\left(\frac{2h_2}{q_3q_4}\right)G(q_1q_2q_3q_4, h_1q_3q_4+h_2q_1q_2, l)\nonumber\\
&\times G(q_1q_2q_3q_4, h_1q_3q_4+h_2q_1q_2, m)\,G(q_1q_2q_3q_4, 0, n+h_1q_3q_4+h_2q_1q_2)\nonumber\\
&=\frac{1}{q_1q_2q_3q_4}\sum\limits_{1\leq h_1\leq q_1q_2}e\left(\frac{h_1}{q_1q_2}\right)
G(q_1q_2, h_1q^2_3q^2_4, l)\, G(q_1q_2, h_1q^2_3q^2_4, m)\, G(q_1q_2, 0, n+h_1q_3q_4)\nonumber\\
&\times\sum\limits_{1\leq h_2\leq q_3q_4}e\left(\frac{2h_2}{q_3q_4}\right)G(q_3q_4, h_2q^2_1q^2_2, l)\,
G(q_3q_4, h_2q^2_1q^2_2, m)\, G(q_3q_4, 0, n+h_2q_1q_2)\nonumber\\
&=\frac{1}{q_1q_2q_3q_4}\sum\limits_{1\leq h_1\leq q_1\atop{1\leq h_2\leq q_2}}e\left(\frac{h_1q_2+h_2q_1}{q_1q_2}\right)
G(q_1q_2, h_1q_2q^2_3q^2_4+h_2q_1q^2_3q^2_4, l)\nonumber\\
&\times G(q_1q_2, h_1q_2q^2_3q^2_4+h_2q_1q^2_3q^2_4, m)\,G(q_1q_2, 0, n+h_1q_2q_3q_4+h_2q_1q_3q_4)\nonumber\\
&\times\sum\limits_{1\leq h_3\leq q_3\atop{1\leq h_4\leq q_4}}e\left(\frac{2(h_3q_4+h_4q_3)}{q_3q_4}\right)
G(q_3q_4, h_3q^2_1q^2_2q_4+h_4q^2_1q^2_2q_3, l)\nonumber\\
&\times G(q_3q_4, h_3q^2_1q^2_2q_4+h_4q^2_1q^2_2q_3, m)\,G(q_3q_4, 0, n+h_3q_1q_2q_4+h_4q_1q_2q_3)\nonumber\\
&=\frac{1}{q_1q_2q_3q_4}\sum\limits_{1\leq h_1\leq q_1\atop{1\leq h_2\leq q_2}}e\left(\frac{h_1q_2+h_2q_1}{q_1q_2}\right)
G(q_1, h_1q^2_2q^2_3q^2_4, l)\,G(q_2, h_2q^2_1q^2_3q^2_4, l)\,G(q_1, h_1q^2_2q^2_3q^2_4, m)\nonumber\\
&\times G(q_2, h_2q^2_1q^2_3q^2_4, m)\,G(q_1, 0, n+h_1q_2q_3q_4)\, G(q_2, 0, n+h_2q_1q_3q_4)\nonumber\\
&\times\sum\limits_{1\leq h_3\leq q_3\atop{1\leq h_4\leq q_4}}e\left(\frac{2(h_3q_4+h_4q_3)}{q_3q_4}\right)
G(q_3, h_3q^2_1q^2_2q^2_4, l)\,G(q_4, h_4q^2_1q^2_2q^2_3, l)\,G(q_3, h_3q^2_1q^2_2q^2_4, m)\nonumber\\
&\times G(q_4, h_4q^2_1q^2_2q^2_3, m)\,G(q_3, 0, n+h_3q_1q_2q_4)\, G(q_4, 0, n+h_4q_1q_2q_3)\,.
\end{align}
On the other hand \eqref{Gausssums}, \eqref{lambdaq1q2lmnast}, \eqref{well-known}, \eqref{2coprimeq1234} and Lemma \ref{Gausslemma} give us
\begin{align*}
&\lambda^\ast\left(q_1, q_3, l\overline{(q_2q_4)}_{q_1q_3}, m\overline{(q_2q_4)}_{q_1q_3}, n\overline{(q_2q_4)}_{q_1q_3}\right)
\lambda^\ast\left(q_2, q_4, l\overline{(q_1 q_3)}_{q_2 q_4}, m\overline{(q_1 q_3)}_{q_2 q_4}, n\overline{(q_1q_3)}_{q_2q_4}\right)\nonumber\\
&=\frac{1}{q_1q_2q_3q_4}\sum\limits_{1\leq x_1, y_1, z_1\leq q_1q_3}
e\left(\frac{l\overline{(q_2q_4)}_{q_1q_3}x_1+m\overline{(q_2q_4)}_{q_1q_3}y_1+ n\overline{(q_2q_4)}_{q_1q_3}z_1}{q_1q_3}\right)\nonumber\\
&\times \sum\limits_{1\leq h_1\leq q_1}e\left(\frac{h_1(x_1^2+y_1^2+z_1^2+z_1+1)}{q_1}\right)
\sum\limits_{1\leq h_3\leq q_3}e\left(\frac{h_3(x_1^2+y_1^2+z_1^2+z_1+2)}{q_3}\right)\nonumber\\
&\times\sum\limits_{1\leq x_2, y_2, z_2\leq q_2q_4}
e\left(\frac{l\overline{(q_1q_3)}_{q_2q_4}x_2+m\overline{(q_1q_3)}_{q_2q_4}y_2+ n\overline{(q_1q_3)}_{q_2q_4}z_2}{q_2q_4}\right)\nonumber\\
\end{align*}

\begin{align}\label{lambdaastq1234lmnest2}
&\times \sum\limits_{1\leq h_2\leq q_2}e\left(\frac{h_2(x_2^2+y_2^2+z_2^2+z_2+1)}{q_2}\right)
\sum\limits_{1\leq h_4\leq q_4}e\left(\frac{h_4(x_2^2+y_2^2+z_2^2+z_2+2)}{q_4}\right)\nonumber\\
&=\frac{1}{q_1q_2q_3q_4}\sum\limits_{1\leq h_1\leq q_1\atop{1\leq h_2\leq q_2}}e\left(\frac{h_1q_2+h_2q_1}{q_1q_2}\right)
\sum\limits_{1\leq h_3\leq q_3\atop{1\leq h_4\leq q_4}}e\left(\frac{2(h_3q_4+h_4q_3)}{q_3q_4}\right)\nonumber\\
&\times G\big(q_1q_3,h_1q_3+h_3q_1, l\overline{(q_2q_4)}_{q_1q_3}\big)\,
G\big(q_1q_3, h_1q_3+h_3q_1, m\overline{(q_2q_4)}_{q_1q_3}\big) \nonumber\\
&\times G\big(q_1q_3, 0, n\overline{(q_2q_4)}_{q_1q_3}+h_1q_3+h_3q_1\big)\,
G\big(q_2q_4, h_2q_4+h_4q_2, l\overline{(q_1q_3)}_{q_2q_4}\big)\nonumber\\
&\times G\big(q_2q_4,h_2q_4+h_4q_2, m\overline{(q_1q_3)}_{q_2q_4}\big)\,
G\big(q_2q_4, 0, n\overline{(q_1q_3)}_{q_2q_4}+h_2q_4+h_4q_2\big)\nonumber\\
&=\frac{1}{q_1q_2q_3q_4}\sum\limits_{1\leq h_1\leq q_1\atop{1\leq h_2\leq q_2}}e\left(\frac{h_1q_2+h_2q_1}{q_1q_2}\right)
G\big(q_1, h_1q^2_3, l\overline{(q_2q_4)}_{q_1q_3}\big)\,G\big(q_2, h_2q^2_4, l\overline{(q_1q_3)}_{q_2q_4}\big)\nonumber\\
&\times G\big(q_1, h_1q^2_3, m\overline{(q_2q_4)}_{q_1q_3}\big)\,G\big(q_2, h_2q^2_4, m\overline{(q_1q_3)}_{q_2q_4}\big)\nonumber\\
&\times G\big(q_1, 0,  n\overline{(q_2q_4)}_{q_1q_3}+h_1q_3\big)\,G\big(q_2, 0, n\overline{(q_1q_3)}_{q_2q_4}+h_2q_4\big)\nonumber\\
&\times\sum\limits_{1\leq h_3\leq q_3\atop{1\leq h_4\leq q_4}}e\left(\frac{2(h_3q_4+h_4q_3)}{q_3q_4}\right)
G\big(q_3, h_3q^2_1, l\overline{(q_2q_4)}_{q_1q_3}\big)\, G\big(q_4, h_4q^2_2, l\overline{(q_1q_3)}_{q_2q_4}\big)\nonumber\\
&\times G\big(q_3, h_3q^2_1, m\overline{(q_2q_4)}_{q_1q_3}\big)\,G\big(q_4, h_4q^2_2, m\overline{(q_1q_3)}_{q_2q_4}\big)\nonumber\\
&\times G\big(q_3, 0, n\overline{(q_2q_4)}_{q_1q_3}+h_3q_1\big)\,G\big(q_4, 0, n\overline{(q_1q_3)}_{q_2q_4}+h_4q_2\big)\,.
\end{align}
Using the substitutions $x\rightarrow \overline{(q_2q_4)}_{q_1q_3}x$, $x\rightarrow \overline{(q_1q_3)}_{q_2q_4}x$
and working as in Lemma \ref{multiplicative1} we deduce
\begin{align}
\label{Gausssum1ast}
&G(q_1, h_1q^2_2q^2_3q^2_4, l)=G\big(q_1, h_1q^2_3, l\overline{(q_2q_4)}_{q_1q_3}\big)\,,\\
\label{Gausssum2ast}
&G(q_2, h_2q^2_1q^2_3q^2_4, l)=G\big(q_2,h_2q^2_4, l\overline{(q_1q_3)}_{q_2q_4}\big)\,,\\
\label{Gausssum3ast}
&G(q_1, h_1q^2_2q^2_3q^2_4, m)=G\big(q_1, h_1q^2_3, m\overline{(q_2q_4)}_{q_1q_3}\big)\,,\\
\label{Gausssum4ast}
&G(q_2, h_2q^2_1q^2_3q^2_4, m)=G\big(q_2,h_2q^2_4, m\overline{(q_1q_3)}_{q_2q_4}\big)\,,\\
\label{Gausssum5ast}
&G(q_1, 0, n+h_1q_2q_3q_4)=G\big(q_1, 0, n\overline{(q_2q_4)}_{q_1q_3}+h_1q_3\big)\,,\\
\label{Gausssum6ast}
&G(q_2, 0, n+h_2q_1q_3q_4)=G\big(q_2, 0, n\overline{(q_1q_3)}_{q_2q_4}+h_2q_4\big)\,,\\
\label{Gausssum7ast}
&G(q_3, h_3q^2_1q^2_2q^2_4, l)=G\big(q_3, h_3q^2_1, l\overline{(q_2q_4)}_{q_1q_3}\big)\,,\\
\label{Gausssum8ast}
&G(q_4, h_4q^2_1q^2_2q^2_3, l)=G\big(q_4, h_4q^2_2, l\overline{(q_1q_3)}_{q_2q_4}\big)\,,\\
\label{Gausssum9ast}
&G(q_3, h_3q^2_1q^2_2q^2_4, m)=G\big(q_3, h_3q^2_1, m\overline{(q_2q_4)}_{q_1q_3}\big)\,,\\
\label{Gausssum10ast}
&G(q_4, h_4q^2_1q^2_2q^2_3, m)=G\big(q_4, h_4q^2_2, m\overline{(q_1q_3)}_{q_2q_4}\big)\,,\\
\label{Gausssum11ast}
&G(q_3, 0, n+h_3q_1q_2q_4)=G\big(q_3, 0, n\overline{(q_2q_4)}_{q_1q_3}+h_3q_1\big)\,,\\
\label{Gausssum12ast}
&G(q_4, 0, n+h_4q_1q_2q_3)=G\big(q_4, 0, n\overline{(q_1q_3)}_{q_2q_4}+h_4q_2\big)\,.
\end{align}
Summarizing \eqref{lambdaastq1234lmnest1} --  \eqref{Gausssum12ast} we complete the proof of the lemma.
\end{proof}

\begin{lemma}\label{lambdaastupperbound}
Let $8\nmid q_1q_2$.
For the function defined by \eqref{lambdaq1q2lmnast} the upper bound
\begin{equation*}
\lambda^\ast(q_1, q_2, l, m, n)\ll q^2_1q^2_2\tau^2(q_1 q_2)
\end{equation*}
holds.
In particular we have
\begin{equation*}
\lambda^\ast(q_1, q_2, l, m, n)\ll (q_1q_2)^{2+\varepsilon}\,.
\end{equation*}
\end{lemma}

\begin{proof}

\textbf{Case 1.} $2\nmid q_1 q_2$.

By \eqref{Gausssums}, \eqref{Kloosterman}, \eqref{lambdaq1q2lmnast}, \eqref{well-known},
Lemma \ref{Gausslemma} and Lemma \ref{Weilsestimate} derive
\begin{align*}
&\lambda^\ast(q_1, q_2, l, m, n)
=\frac{1}{q_1q_2}\sum\limits_{1\leq x, y,z\leq q_1q_2}e\left(\frac{lx+my+nz}{q_1q_2}\right)\nonumber\\
&\times\sum\limits_{1\leq h_1\leq q_1}e\left(\frac{h_1(x^2+y^2+z+1)}{q_1}\right)
\sum\limits_{1\leq h_2\leq q_2}e\left(\frac{h_2(x^2+y^2+z+1)}{q_2}\right)\nonumber\\
&=\frac{1}{q_1q_2}\sum\limits_{1\leq z\leq q_1q_2}e\left(\frac{nz}{q_1q_2}\right)
\sum\limits_{1\leq h_1\leq q_1}e\left(\frac{h_1(z+1)}{q_1}\right)
\sum\limits_{1\leq h_2\leq q_2}e\left(\frac{2h_2(z+1)}{q_2}\right) \nonumber\\
&\times G(q_1 q_2, h_1 q_2+h_2 q_1, l)\, G(q_1 q_2, h_1 q_2+h_2 q_1, m) \nonumber\\
&=\frac{1}{q_1 q_2}\sum\limits_{1\leq z\leq q_1q_2}e\left(\frac{nz}{q_1q_2}\right)
\sum\limits_{1\leq h_1\leq q_1}e\left(\frac{h_1(z+1)}{q_1}\right)
G(q_1, h_1 q^2_2, l)\, G(q_1, h_1 q^2_2, m)\nonumber\\
&\times\sum\limits_{1\leq h_2\leq q_2}e\left(\frac{2h_2(z+1)}{q_2}\right)
G(q_2, h_2 q^2_1, l)\, G(q_2, h_2 q^2_1, m)\nonumber\\
&=\frac{1}{q_1 q_2}\sum\limits_{1\leq z\leq q_1q_2}e\left(\frac{nz}{q_1q_2}\right)
\sum\limits_{d_1 | q_1}\sum\limits_{1\leq h_1\leq q_1\atop{(h_1, q_1)
=\frac{q_1}{d_1}}}e\left(\frac{h_1(z+1)}{q_1}\right)G(q_1, h_1 q^2_2, l)\, G(q_1, h_1 q^2_2, m)\nonumber\\
&\times\sum\limits_{d_2 | q_2}\sum\limits_{1\leq h_2\leq q_2\atop{(h_2, q_2)
=\frac{q_2}{d_2}}}e\left(\frac{2h_2(z+1)}{q_2}\right)G(q_2, h_2 q^2_1, l)\, G(q_2, h_2 q^2_1, m)\nonumber\\
&=q_1 q_2\sum\limits_{1\leq z\leq q_1q_2}e\left(\frac{nz}{q_1q_2}\right)
\sum\limits_{d_1 | q_1\atop{\frac{q_1}{d_1} | (l, m)}}\frac{1}{d^2_1}
\sum\limits_{1\leq r_1\leq d_1\atop{(r_1, d_1)=1}}e\left(\frac{r_1(z+1)}{d_1}\right)
G(d_1, r_1 q^2_2, l d_1 q^{-1}_1) G(d_1, r_1 q^2_2, m d_1 q^{-1}_1) \nonumber\\
&\times\sum\limits_{d_2 | q_2\atop{\frac{q_2}{d_2} | (l, m)}}\frac{1}{d^2_2}
\sum\limits_{1\leq r_2\leq d_2\atop{(r_2, d_2)=1}}e\left(\frac{2r_2(z+1)}{d_2}\right)
G(d_2, r_2 q^2_1, l d_2 q^{-1}_2) G(d_2, r_2 q^2_1, m d_2 q^{-1}_2) \nonumber\\
\end{align*}

\begin{align}\label{lambdaastq1q2lmnest1}
&=q_1 q_2\sum\limits_{1\leq z\leq q_1q_2}e\left(\frac{nz}{q_1q_2}\right)
\sum\limits_{d_1 | q_1\atop{\frac{q_1}{d_1} | (l, m)}}\frac{G^2(d_1, 1)}{d^2_1}
\sum\limits_{1\leq r_1\leq d_1\atop{(r_1, d_1)=1}}
e\left(\frac{r_1(z+1)-\overline{(4 r_1 q^2_2)}_{d_1}(l^2+m^2)d^2_1 q^{-2}_1}{d_1}\right)\nonumber\\
&\times\sum\limits_{d_2 | q_2\atop{\frac{q_2}{d_2} | (l, m)}}\frac{G^2(d_2, 1)}{d^2_2}
\sum\limits_{1\leq r_2\leq d_2\atop{(r_2, d_2)=1}}
e\left(\frac{2r_2(z+1)-\overline{(4 r_2 q^2_1)}_{d_2}(l^2+m^2)d^2_2 q^{-2}_2}{d_2}\right)\nonumber\\
&=q_1 q_2\sum\limits_{1\leq z\leq q_1q_2}e\left(\frac{nz}{q_1q_2}\right)
\sum\limits_{d_1 | q_1\atop{\frac{q_1}{d_1} | (l, m)}}\frac{G^2(d_1, 1)}{d^2_1}
K\big(d_1, z+1, -\overline{(4q^2_2)}_{d_1}(l^2+m^2)d^2_1 q^{-2}_1\big)\nonumber\\
&\times\sum\limits_{d_2 | q_2\atop{\frac{q_2}{d_2} | (l, m)}}\frac{G^2(d_2, 1)}{d^2_2}
K\big(d_2, z+1, -\overline{(4q^2_1)}_{d_2}(l^2+m^2)d^2_2 q^{-2}_2\big)\nonumber\\
&\ll q_1 q_2\sum\limits_{1\leq z\leq q_1 q_2}
\sum\limits_{d_1 | q_1\atop{\frac{q_1}{d_1} | (l, m)}}
\tau(d_1)d_1^{-\frac{1}{2}}\big(d_1, z+1,\overline{4}_{d_1}(l^2+m^2)d^2_1 q^{-2}_1\big)^{\frac{1}{2}}\nonumber\\
&\times\sum\limits_{d_2 | q_2\atop{\frac{q_2}{d_2} | (l, m)}}
\tau(d_2)d_2^{-\frac{1}{2}}\big(d_2, z+1,\overline{4}_{d_2}(l^2+m^2)d^2_2 q^{-2}_2\big)^{\frac{1}{2}}\nonumber\\
&\ll q_1 q_2\sum\limits_{1\leq z\leq q_1 q_2}\sum\limits_{d_1 | q_1\atop{\frac{q_1}{d_1} | (l, m)}}\tau(d_1)
\sum\limits_{d_2 | q_2\atop{\frac{q_2}{d_2} | (l, m)}}\tau(d_2)
\ll q^2_1 q^2_2\tau(q_1q_2)\sum\limits_{r_1 | (q_1, l, m)}\sum\limits_{r_2 | (q_2, l, m)}1\nonumber\\
&\ll q^2_1 q^2_2\tau^2(q_1q_2)\,.
\end{align}

\textbf{Case 2.} $q_1=2^h q'_1$, where $2\nmid q'_1$, $h\leq2$ and $2\nmid q_2$.

Using \eqref{lambdaastq1q2lmnest1}, Lemma \ref{multiplicative2} and the
trivial estimate $|\lambda^\ast(2^h, 1, l, m, n)|\leq8^h$ we find
\begin{align}\label{lambdaastq1q2lmnest2}
\lambda^\ast(2^h q'_1, q_2, l, m, n)&=
\lambda^\ast\left(2^h, 1, l\overline{(q'_1 q_2)}_{2^h}, m \overline{(q'_1 q_2)}_{2^h}, n \overline{(q'_1 q_2)}_{2^h}\right)\nonumber\\
&\times\lambda^\ast\left(q'_1, q_2, l \overline{(2^h)}_{q'_1 q_2}, m \overline{(2^h)}_{q'_1 q_2}, n \overline{(2^h)}_{q'_1 q_2}\right)\nonumber\\
&\ll q'^2_1 q^2_2\tau^2(q'_1 q_2)\ll q^2_1 q^2_2\tau^2(q_1 q_2)\,.
\end{align}

\textbf{Case 3.} $q_2=2^h q'_2$, where $2\nmid q'_2$, $h\leq2$ and $2\nmid q_1$.

By \eqref{lambdaastq1q2lmnest1}, Lemma \ref{multiplicative2} and the trivial estimate $|\lambda^\ast(1, 2^h, l, m, n)|\leq8^h$ we obtain
\begin{align}\label{lambdaastq1q2lmnest3}
\lambda^\ast(q_1, 2^h q'_2, l, m, n)&=
\lambda^\ast\left(1, 2^h, l\overline{(q_1 q'_2)}_{2^h}, m \overline{(q_1 q'_2)}_{2^h}, n \overline{(q_1 q'_2)}_{2^h}\right)\nonumber\\
&\times\lambda^\ast\left(q_1, q'_2, l\overline{(2^h)}_{q_1 q'_2}, m\overline{(2^h)}_{q_1q'_2}, n\overline{(2^h)}_{q_1q'_2}\right)\nonumber\\
&\ll q^2_1 q'^2_2\tau^2(q_1 q'_2)\ll q^2_1 q^2_2\tau^2(q_1 q_2)\,.
\end{align}
Now the lemma follows from \eqref{lambdaastq1q2lmnest1} -- \eqref{lambdaastq1q2lmnest3}.
\end{proof}

\begin{lemma}\label{Lambdaast123est}
Assume that $8\nmid q_1 q_2$ and $H_0\geq2$. Then for the sums
\begin{equation}\label{Lambdaast12}
\Lambda^\ast_1=\sum\limits_{1\leq l\leq H_0}\frac{|\lambda^\ast(q_1, q_2, l, 0, 0)|}{l}\,,\quad
\Lambda^\ast_2=\sum\limits_{1\leq n\leq H_0}\frac{|\lambda^\ast(q_1, q_2, 0, 0, n)|}{n}\,,
\end{equation}

\begin{equation}\label{Lambdaast34}
\Lambda^\ast_3=\sum\limits_{1\leq l, m\leq H_0}\frac{|\lambda^\ast(q_1, q_2, l, m, 0)|}{l m}\,,\quad
\Lambda^\ast_4=\sum\limits_{1\leq l, n\leq H_0}\frac{|\lambda^\ast(q_1, q_2, l, 0, n)|}{l n}
\end{equation}
and
\begin{equation}\label{Lambdaast5}
\Lambda^\ast_5=\sum\limits_{1\leq l, m, n\leq H_0}\frac{|\lambda^\ast(q_1, q_2, l, m, n)|}{l m n}
\end{equation}
the estimations
\begin{equation*}%\label{Lambda12est}
\Lambda^\ast_i\ll (q_1q_2)^{2+\varepsilon} H^\varepsilon_0\,,\quad i=1, 2, 3, 4, 5
\end{equation*}
hold.
\end{lemma}
\begin{proof}
Using \eqref{Lambdaast12} and Lemma \ref{lambdaastupperbound} we get
\begin{equation*}
\Lambda^\ast_1, \Lambda^\ast_2\ll (q_1q_2)^{2+\varepsilon}
\sum\limits_{1\leq l\leq H_0}\frac{1}{l}\ll (q_1q_2)^{2+\varepsilon}H^\varepsilon_0\,
\end{equation*}
Further \eqref{Lambdaast34} and Lemma \ref{lambdaastupperbound} yield
\begin{equation*}
\Lambda^\ast_3, \Lambda^\ast_4\ll (q_1q_2)^{2+\varepsilon}
\left(\sum\limits_{1\leq l\leq H_0}\frac{1}{l}\right)^2\ll (q_1q_2)^{2+\varepsilon}H^\varepsilon_0\,
\end{equation*}
Finally \eqref{Lambdaast5} and Lemma \ref{lambdaastupperbound} imply
\begin{equation*}
\Lambda^\ast_5\ll (q_1q_2)^{2+\varepsilon}
\left(\sum\limits_{1\leq l\leq H_0}\frac{1}{l}\right)^3\ll (q_1q_2)^{2+\varepsilon}H^\varepsilon_0\,.
\end{equation*}
\end{proof}

\section{Proof of Theorem 1}\label{ProofofTheorem1}
\indent

Using \eqref{GammaH} and the well-known identity
\begin{equation}\label{wellknownidentity}
\mu^2(n)=\sum\limits_{d^2|n}\mu(d)
\end{equation}
we write
\begin{equation}\label{GammaHdecomp}
\Gamma(H)=\sum\limits_{d_1, d_2\atop{(d_1, d_2)=1}}\mu(d_1)\mu(d_2)
\sum\limits_{1\leq x, y, z\leq H\atop{x^2+y^2+z^2+z+1\equiv 0\,(d_1^2)\atop{x^2+y^2+z^2+z+2\equiv 0\,(d_2^2)}}}1
=\Gamma_1(H)+\Gamma_2(H)\,,
\end{equation}
where
\begin{align}
\label{GammaH1}
&\Gamma_1(H)=\sum\limits_{d_1 d_2\leq \xi\atop{(d_1, d_2)=1}}\mu(d_1)\mu(d_2)\Sigma(H, d_1^2, d_2^2)\,,\\
\label{GammaH2}
&\Gamma_2(H)=\sum\limits_{d_1 d_2>\xi\atop{(d_1, d_2)=1}}\mu(d_1)\mu(d_2)\Sigma(H, d_1^2, d_2^2)\,,\\
\label{Sigma}
&\Sigma(H, d_1^2, d_2^2)=\sum\limits_{1\leq x, y, z\leq H\atop{x^2+y^2+z^2+z+1\equiv 0\,(d_1^2)\atop{x^2+y^2+z^2+z+2\equiv 0\,(d_2^2)}}}1\,,\\
\label{xiH}
&\sqrt{H}\leq \xi\leq H\,.
\end{align}

\subsection{Estimation of $\mathbf{\Gamma_1(H)}$}\label{EstimationofGammaH1}

In this subsection we follow the method in Fan and Zhai \cite{Fan}.
At the estimation of $\Gamma_1(H)$ we will suppose that $q_1=d_1^2$, $q_2=d_2^2$, where
$d_1$ and $d_2$ are square-free, $(q_1, q_2)=1$ and $d_1 d_2\leq \xi$.
Denote
\begin{equation}\label{Omegaq1q2}
\Omega(H, q_1, q_2, x)=\sum\limits_{1\leq h\leq H\atop{h\equiv x\,(q_1q_2)}}1\,.
\end{equation}
Now \eqref{well-known} and \eqref{Omegaq1q2} lead to
\begin{align}\label{Omegaq1q2est}
\Omega(H, q_1, q_2, x)&=\frac{1}{q_1q_2}\sum\limits_{1\leq h\leq H}
\sum\limits_{t=1}^{q_1q_2}e\left(\frac{(h-x)t}{q_1q_2}\right)\nonumber\\
&=\frac{1}{q_1q_2}\sum\limits_{t=1}^{q_1q_2}e\left(-\frac{xt}{q_1q_2}\right)
\sum\limits_{1\leq h\leq H}e\left(\frac{ht}{q_1q_2}\right)\nonumber\\
&=\frac{H}{q_1q_2}+\frac{1}{q_1q_2}\sum\limits_{t=1}^{q_1q_2-1}e\left(-\frac{xt}{q_1q_2}\right)
\sum\limits_{1\leq h\leq H}e\left(\frac{ht}{q_1q_2}\right)\,.
\end{align}
From \eqref{Sigma}, \eqref{Omegaq1q2} and \eqref{Omegaq1q2est} we obtain
\begin{align*}
\Sigma(H, q_1, q_2)
&=\sum\limits_{1\leq x, y, z\leq q_1q_2\atop{x^2+y^2+z^2+z+1\equiv 0\,(q_1)\atop{x^2+y^2+z^2+z+2\equiv 0\,(q_2)}}}
\Omega(H, q_1, q_2, x)\,\Omega(H, q_1, q_2, y)\,\Omega(H, q_1, q_2, z)\nonumber\\
\end{align*}
\begin{align}\label{Sigmaq1q2est1}
&=\frac{1}{(q_1q_2)^3}\sum\limits_{1\leq x, y, z\leq q_1q_2\atop{x^2+y^2+z^2+z+1\equiv 0\,(q_1)\atop{x^2+y^2+z^2+z+2\equiv 0\,(q_2)}}}
\Big(H^3+2H^2S_1+H^2S'_1+2HS_2+HS'_2+S_3\Big)\,,
\end{align}
where
\begin{align}
\label{S1q1q2}
&S_1:=S_1(x, q_1, q_2, H)=\sum\limits_{t=1}^{q_1q_2-1}e\left(-\frac{xt}{q_1q_2}\right)
\sum\limits_{1\leq h\leq H}e\left(\frac{ht}{q_1q_2}\right)\,,\\
\label{S'1q1q2}
&S'_1:=S_1(z, q_1, q_2, H)=\sum\limits_{t=1}^{q_1q_2-1}e\left(-\frac{zt}{q_1q_2}\right)
\sum\limits_{1\leq h\leq H}e\left(\frac{ht}{q_1q_2}\right)\,,\\
\label{S2q1q2}
&S_2:=S_2(x, z, q_1, q_2, H)=\sum\limits_{t_1=1}^{q_1q_2-1}\sum\limits_{t_2=1}^{q_1q_2-1}e\left(-\frac{xt_1+zt_2}{q_1q_2}\right)
\prod\limits_{i=1}^{2}\sum\limits_{1\leq h_i\leq H}e\left(\frac{h_it_i}{q_1q_2}\right)\,,\\
\label{S'2q1q2}
&S'_2:=S_2(x, y, q_1, q_2, H)=\sum\limits_{t_1=1}^{q_1q_2-1}\sum\limits_{t_2=1}^{q_1q_2-1}e\left(-\frac{xt_1+yt_2}{q_1q_2}\right)
\prod\limits_{i=1}^{2}\sum\limits_{1\leq h_i\leq H}e\left(\frac{h_it_i}{q_1q_2}\right)\,,\\
\label{S3q1q2}
&S_3:=S_3(x, y, z, q_1, q_2, H)=\sum\limits_{t_1=1}^{q_1q_2-1}\sum\limits_{t_2=1}^{q_1q_2-1}
\sum\limits_{t_3=1}^{q_1q_2-1}e\left(-\frac{xt_1+yt_2+zt_3}{q_1q_2}\right)
\prod\limits_{i=1}^{3}\sum\limits_{1\leq h_i\leq H}e\left(\frac{h_it_i}{q_1q_2}\right)\,.
\end{align}
Taking into account \eqref{lambdaq1q2}, \eqref{Sigmaq1q2est1} -- \eqref{S3q1q2} we derive
\begin{equation}\label{Sigmaq1q2est2}
\Sigma(H, q_1, q_2)=\frac{1}{(q_1q_2)^3}\Big(H^3\lambda(q_1, q_2)+2H^2\Sigma_1+H^2\Sigma'_1+2H\Sigma_2+H\Sigma'_2+\Sigma_3\Big)\,,
\end{equation}
where
\begin{align}
\label{Sigma1q1q2}
&\Sigma_1=\sum\limits_{1\leq x, y, z\leq q_1q_2\atop{x^2+y^2+z^2+z+1\equiv 0\,(q_1)\atop{x^2+y^2+z^2+z+2\equiv 0\,(q_2)}}}
\sum\limits_{t=1}^{q_1q_2-1}e\left(-\frac{xt}{q_1q_2}\right)\sum\limits_{1\leq h\leq H}e\left(\frac{ht}{q_1q_2}\right)\,,\\
\label{Sigma'1q1q2}
&\Sigma'_1=\sum\limits_{1\leq x, y, z\leq q_1q_2\atop{x^2+y^2+z^2+z+1\equiv 0\,(q_1)\atop{x^2+y^2+z^2+z+2\equiv 0\,(q_2)}}}
\sum\limits_{t=1}^{q_1q_2-1}e\left(-\frac{zt}{q_1q_2}\right)\sum\limits_{1\leq h\leq H}e\left(\frac{ht}{q_1q_2}\right)\,,\\
\label{Sigma2q1q2}
&\Sigma_2=\sum\limits_{1\leq x, y, z\leq q_1q_2\atop{x^2+y^2+z^2+z+1\equiv 0\,(q_1)\atop{x^2+y^2+z^2+z+2\equiv 0\,(q_2)}}}
\sum\limits_{t_1=1}^{q_1q_2-1}\sum\limits_{t_2=1}^{q_1q_2-1}e\left(-\frac{xt_1+zt_2}{q_1q_2}\right)
\prod\limits_{i=1}^{2}\sum\limits_{1\leq h_i\leq H}e\left(\frac{h_it_i}{q_1q_2}\right)\,,\\
\label{Sigma'2q1q2}
&\Sigma'_2=\sum\limits_{1\leq x, y, z\leq q_1q_2\atop{x^2+y^2+z^2+z+1\equiv 0\,(q_1)\atop{x^2+y^2+z^2+z+2\equiv 0\,(q_2)}}}
\sum\limits_{t_1=1}^{q_1q_2-1}\sum\limits_{t_2=1}^{q_1q_2-1}e\left(-\frac{xt_1+yt_2}{q_1q_2}\right)
\prod\limits_{i=1}^{2}\sum\limits_{1\leq h_i\leq H}e\left(\frac{h_it_i}{q_1q_2}\right)\,,
\end{align}

\begin{align}
\label{Sigma3q1q2}
&\Sigma_3=\sum\limits_{1\leq x, y, z\leq q_1q_2\atop{x^2+y^2+z^2+z+1\equiv 0\,(q_1)\atop{x^2+y^2+z^2+z+2\equiv 0\,(q_2)}}}
\sum\limits_{t_1=1}^{q_1q_2-1}\sum\limits_{t_2=1}^{q_1q_2-1}\sum\limits_{t_3=1}^{q_1q_2-1}e\left(-\frac{xt_1+yt_2+zt_3}{q_1q_2}\right)
\prod\limits_{i=1}^{3}\sum\limits_{1\leq h_i\leq H}e\left(\frac{h_it_i}{q_1q_2}\right)\,.
\end{align}
By \eqref{lambdaq1q2lmn}, \eqref{Sigma1q1q2}, Lemma \ref{minimum} and Lemma \ref{Lambda123est} it follows
\begin{align}\label{Sigma1q1q2est1}
\Sigma_1&=\sum\limits_{t=1}^{q_1q_2-1}\lambda(q_1, q_2, -t,0, 0)\sum\limits_{1\leq h\leq H}e\left(\frac{ht}{q_1q_2}\right)
\ll \sum\limits_{t=1}^{q_1q_2-1}|\lambda(q_1, q_2, -t,0, 0)|  \left\|\frac{t}{q_1q_2}\right\|^{-1}\nonumber\\
&\ll q_1q_2\sum\limits_{t=1}^{q_1q_2-1}\frac{|\lambda(q_1, q_2, -t,0, 0)| }{t}\ll (q_1q_2)^{2+\varepsilon}\,.
\end{align}
Arguing in the same way for the sums \eqref{Sigma'1q1q2} -- \eqref{Sigma3q1q2} we deduce
\begin{equation}\label{Sigma23q1q2est1}
\Sigma'_1\ll (q_1q_2)^{2+\varepsilon}\,, \quad \Sigma_2\ll (q_1q_2)^{3+\varepsilon}\,, \quad
\Sigma'_2\ll (q_1q_2)^{3+\varepsilon}\,, \quad \Sigma_3 \ll (q_1q_2)^{4+\varepsilon}\,.
\end{equation}
Now \eqref{Sigmaq1q2est2}, \eqref{Sigma1q1q2est1} and \eqref{Sigma23q1q2est1} give us
\begin{equation}\label{Sigmaest3}
\Sigma(H, q_1, q_2)=\frac{H^3}{(q_1q_2)^3}\lambda(q_1, q_2)+\mathcal{O}\Big(H^2(q_1q_2)^{\varepsilon-1}
+ H(q_1q_2)^\varepsilon+(q_1q_2)^{1+\varepsilon}\Big)\,.
\end{equation}
Bearing in mind \eqref{GammaH1}, \eqref{xiH} and \eqref{Sigmaest3} we get
\begin{align}\label{GammaH1est1}
\Gamma_1(H)&=H^3  \sum\limits_{d_1 d_2\leq \xi\atop{(d_1, d_2)=1}}\frac{\mu(d_1)\mu(d_2)\lambda(d^2_1, d^2_2)}{d^6_1 d^6_2}\nonumber\\
&+\mathcal{O}\Bigg(\sum\limits_{d_1 d_2\leq \xi}\Big(H^2(d_1d_2)^{\varepsilon-2}
+H(d_1d_2)^\varepsilon+(d_1d_2)^{2+\varepsilon}\Big)\Bigg)\nonumber\\
&=\sigma H^3-H^3 \sum\limits_{d_1 d_2>\xi}\frac{\mu(d_1)\mu(d_2)\lambda(d^2_1, d^2_2)}{d^6_1 d^6_2}
+\mathcal{O}\Big(H^2+ \xi^{3+\varepsilon}\Big)\,,
\end{align}
where
\begin{equation}\label{sigmasum}
\sigma=\sum\limits_{d_1, d_2=1\atop{(d_1, d_2)=1}}^\infty\frac{\mu(d_1)\mu(d_2)\lambda(d^2_1, d^2_2)}{d^6_1 d^6_2}\,.
\end{equation}
From Lemma \ref{lambdaupperbound} we find
\begin{equation}\label{d1d2>est}
\sum\limits_{d_1d_2>\xi}\frac{\mu(d_1)\mu(d_2)\lambda(d^2_1, d^2_2)}{d^6_1 d^6_2}
\ll\sum\limits_{d_1d_2>\xi}\frac{(d_1d_2)^{4+\varepsilon}}{(d_1d_2)^6}
\ll\sum\limits_{n>\xi}\frac{\tau(n)}{n^{2-\varepsilon}}\ll \xi^{\varepsilon-1}\,.
\end{equation}
Using \eqref{lambdaq1q2}, \eqref{sigmasum}, Lemma \ref{multiplicative1}, the Euler product and
arguing as in \cite{Dimitrov1} we obtain
\begin{equation}\label{productp6}
\sigma=\prod\limits_{p}\left(1-\frac{\lambda(p^2, 1)+\lambda(1, p^2)}{p^6}\right)\,.
\end{equation}
Now \eqref{GammaH1est1}, \eqref{d1d2>est} and \eqref{productp6} yield
\begin{equation}\label{GammaH1est2}
\Gamma_1(H)=\prod\limits_{p}\left(1-\frac{\lambda(p^2, 1)+\lambda(1, p^2)}{p^6}\right)H^3
+\mathcal{O}\Big(H^3\xi^{\varepsilon-1}+H^2+ \xi^{3+\varepsilon}\Big)\,.
\end{equation}

\subsection{Estimation of $\mathbf{\Gamma_2(H)}$}\label{SubsectionGammaH2est}

Using \eqref{GammaH2} we write
\begin{equation}\label{GammaH2est1}
\Gamma_2(H)\ll(\log H)^2\sum\limits_{D_1\leq d_1<2D_1}\sum\limits_{D_2\leq d_2<2D_2}
\sum\limits_{k\leq(3H^2+H+1)d_1^{-2}\atop{kd_1^2+1\equiv0\,(d_2^2)}}
\sum\limits_{1\leq x, y, z\leq H\atop{x^2+y^2+z^2+z=kd_1^2-1}}1\,,
\end{equation}
where
\begin{equation}\label{DT}
\frac{1}{2}\leq D_1, D_2\leq\sqrt{3H^2+H+2}\,,\quad D_1D_2\geq\frac{\xi}{4}\,.
\end{equation}
On the one hand \eqref{GammaH2est1} gives us
\begin{align}\label{GammaH2est2}
\Gamma_2(H)&\ll H^\varepsilon\sum\limits_{D_1\leq d_1<2D_1}
\sum\limits_{k\leq(3H^2+H+1)D_1^{-2}}\sum\limits_{D_2\leq d_2<2D_2}
\sum\limits_{l\leq(3H^2+H+2)D_2^{-2}\atop{kd_1^2+1=ld_2^2}}1\nonumber\\
&\ll  H^\varepsilon\sum\limits_{D_1\leq d_1<2D_1}
\sum\limits_{k\leq(3H^2+H+1)D_1^{-2}}\tau(kd_1^2+1)\nonumber\\
&\ll H^\varepsilon\sum\limits_{D_1\leq d_1<2D_1}
\sum\limits_{k\leq(3H^2+H+1)D_1^{-2}}1\nonumber\\
&\ll H^{2+\varepsilon}D_1^{-1}\,.
\end{align}
On the other hand \eqref{GammaH2est1} implies
\begin{align}\label{GammaH2est3}
\Gamma_2(H)&\ll H^\varepsilon\sum\limits_{D_2\leq d_2<2D_2}
\sum\limits_{l\leq(3H^2+H+2)D_2^{-2}}\sum\limits_{D_1\leq d_1<2D_1}
\sum\limits_{k\leq(3H^2+H+1)D_1^{-2}\atop{kd_1^2=ld_2^2-1}}1\nonumber\\
&\ll H^\varepsilon\sum\limits_{D_2\leq d_2<2D_2}
\sum\limits_{l\leq(3H^2+H+2)D_2^{-2}}\tau(ld_2^2-1)\nonumber\\
&\ll H^\varepsilon\sum\limits_{D_2\leq d_2<2D_2}
\sum\limits_{l\leq(3H^2+H+2)D_2^{-2}}1\nonumber\\
&\ll H^{2+\varepsilon}D_2^{-1}\,.
\end{align}
By \eqref{DT} -- \eqref{GammaH2est3} it follows
\begin{equation}\label{GammaH2est4}
\Gamma_2(H)\ll  H^{2+\varepsilon}\xi^{-\frac{1}{2}}\,.
\end{equation}

\subsection{The end of the proof of Theorem 1}

Summarizing \eqref{GammaHdecomp}, \eqref{GammaH1est2}, \eqref{GammaH2est4}
and choosing $\xi=H^\frac{3}{4}$ we establish the asymptotic formula \eqref{asymptoticformula1}.
This completes the proof of Theorem \ref{Theorem1}.

\section{Proof of Theorem 2}\label{ProofofTheorem2}
\indent

By \eqref{GammaHast} and \eqref{wellknownidentity} we deduce
\begin{equation}\label{GammaHastdecomp}
\Gamma^\ast(H)=\sum\limits_{d_1, d_2\atop{(d_1, d_2)=1}}\mu(d_1)\mu(d_2)
\sum\limits_{1\leq x, y, z\leq H\atop{x^2+y^2+z+1\equiv 0\,(d_1^2)\atop{x^2+y^2+z+2\equiv 0\,(d_2^2)}}}1
=\Gamma^\ast_1(H)+\Gamma^\ast_2(H)\,,
\end{equation}
where
\begin{align}
\label{GammaHast1}
&\Gamma^\ast_1(H)=\sum\limits_{d_1 d_2\leq \xi^\ast\atop{(d_1, d_2)=1}}\mu(d_1)\mu(d_2)\Sigma(H, d_1^2, d_2^2)\,,\\
\label{GammaHast2}
&\Gamma^\ast_2(H)=\sum\limits_{d_1 d_2>\xi^\ast\atop{(d_1, d_2)=1}}\mu(d_1)\mu(d_2)\Sigma(H, d_1^2, d_2^2)\,,\\
\label{Sigmaast}
&\Sigma^\ast(H, d_1^2, d_2^2)=\sum\limits_{1\leq x, y, z\leq H\atop{x^2+y^2+z+1\equiv 0\,(d_1^2)\atop{x^2+y^2+z+2\equiv 0\,(d_2^2)}}}1\,,\\
\label{xiHast}
&\sqrt{H}\leq \xi^\ast\leq H\,.
\end{align}

\subsection{Estimation of $\mathbf{\Gamma^\ast_1(H)}$}

Using \eqref{lambdaq1q2ast}, \eqref{Sigmaast} and proceeding as in Subsection \ref{EstimationofGammaH1} we get
\begin{equation}\label{Sigmaastest1}
\Sigma^\ast(H, q_1, q_2)=\frac{1}{(q_1q_2)^3}\Big(H^3\lambda^\ast(q_1, q_2)+2H^2\Sigma^\ast_1
+H^2\widehat{\Sigma}^\ast_1+2H\Sigma^\ast_2+H\widehat{\Sigma}^\ast_2+\Sigma^\ast_3\Big)\,,
\end{equation}
where
\begin{align}
\label{Sigmaast1q1q2}
&\Sigma^\ast_1=\sum\limits_{1\leq x, y, z\leq q_1q_2\atop{x^2+y^2+z+1\equiv 0\,(q_1)\atop{x^2+y^2+z+2\equiv 0\,(q_2)}}}
\sum\limits_{t=1}^{q_1q_2-1}e\left(-\frac{xt}{q_1q_2}\right)\sum\limits_{1\leq h\leq H}e\left(\frac{ht}{q_1q_2}\right)\,,\\
\label{Sigma'ast1q1q2}
&\widehat{\Sigma}^\ast_1=\sum\limits_{1\leq x, y, z\leq q_1q_2\atop{x^2+y^2+z+1\equiv 0\,(q_1)\atop{x^2+y^2+z+2\equiv 0\,(q_2)}}}
\sum\limits_{t=1}^{q_1q_2-1}e\left(-\frac{zt}{q_1q_2}\right)\sum\limits_{1\leq h\leq H}e\left(\frac{ht}{q_1q_2}\right)\,,
\end{align}

\begin{align}
\label{Sigmaast2q1q2}
&\Sigma^\ast_2=\sum\limits_{1\leq x, y, z\leq q_1q_2\atop{x^2+y^2+z+1\equiv 0\,(q_1)\atop{x^2+y^2+z+2\equiv 0\,(q_2)}}}
\sum\limits_{t_1=1}^{q_1q_2-1}\sum\limits_{t_2=1}^{q_1q_2-1}e\left(-\frac{xt_1+zt_2}{q_1q_2}\right)
\prod\limits_{i=1}^{2}\sum\limits_{1\leq h_i\leq H}e\left(\frac{h_it_i}{q_1q_2}\right)\,,\\
\label{Sigma'ast2q1q2}
&\widehat{\Sigma}^\ast_2=\sum\limits_{1\leq x, y, z\leq q_1q_2\atop{x^2+y^2+z+1\equiv 0\,(q_1)\atop{x^2+y^2+z+2\equiv 0\,(q_2)}}}
\sum\limits_{t_1=1}^{q_1q_2-1}\sum\limits_{t_2=1}^{q_1q_2-1}e\left(-\frac{xt_1+yt_2}{q_1q_2}\right)
\prod\limits_{i=1}^{2}\sum\limits_{1\leq h_i\leq H}e\left(\frac{h_it_i}{q_1q_2}\right)\,,\\
\label{Sigmaast3q1q2}
&\Sigma^\ast_3=\sum\limits_{1\leq x, y, z\leq q_1q_2\atop{x^2+y^2+z+1\equiv 0\,(q_1)\atop{x^2+y^2+z+2\equiv 0\,(q_2)}}}
\sum\limits_{t_1=1}^{q_1q_2-1}\sum\limits_{t_2=1}^{q_1q_2-1}\sum\limits_{t_3=1}^{q_1q_2-1}e\left(-\frac{xt_1+yt_2+zt_3}{q_1q_2}\right)
\prod\limits_{i=1}^{3}\sum\limits_{1\leq h_i\leq H}e\left(\frac{h_it_i}{q_1q_2}\right)\,.
\end{align}
Now \eqref{lambdaq1q2lmnast}, \eqref{Sigmaast1q1q2}, Lemma \ref{minimum} and Lemma \ref{Lambdaast123est} lead to
\begin{align}\label{Sigmaast1est}
\Sigma^\ast_1&=\sum\limits_{t=1}^{q_1q_2-1}\lambda^\ast(q_1, q_2, -t,0, 0)\sum\limits_{1\leq h\leq H}e\left(\frac{ht}{q_1q_2}\right)
\ll \sum\limits_{t=1}^{q_1q_2-1}|\lambda^\ast(q_1, q_2, -t,0, 0)|  \left\|\frac{t}{q_1q_2}\right\|^{-1}\nonumber\\
&\ll q_1q_2\sum\limits_{t=1}^{q_1q_2-1}\frac{|\lambda^\ast(q_1, q_2, -t,0, 0)| }{t}\ll (q_1q_2)^{3+\varepsilon}\,.
\end{align}
Working in the same way for the sums \eqref{Sigma'ast1q1q2} --\eqref{Sigmaast3q1q2} we find
\begin{equation}\label{Sigmaast23est}
\widehat{\Sigma}^\ast_1\ll (q_1q_2)^{3+\varepsilon}\,, \quad \Sigma^\ast_2\ll (q_1q_2)^{4+\varepsilon}\,, \quad
\widehat{\Sigma}^\ast_2\ll (q_1q_2)^{4+\varepsilon}\,, \quad \Sigma^\ast_3 \ll (q_1q_2)^{5+\varepsilon}\,.
\end{equation}
Further \eqref{Sigmaastest1}, \eqref{Sigmaast1est} and \eqref{Sigmaast23est} give us
\begin{equation}\label{Sigmaastest2}
\Sigma^\ast(H, q_1, q_2)=\frac{H^3}{(q_1q_2)^3}\lambda^\ast(q_1, q_2)+\mathcal{O}\Big(H^2(q_1q_2)^\varepsilon
+ H(q_1q_2)^{1+\varepsilon}+(q_1q_2)^{2+\varepsilon}\Big)\,.
\end{equation}
Taking into account \eqref{GammaHast1}, \eqref{xiHast} and \eqref{Sigmaastest2} we  deduce
\begin{align}\label{GammaHast1est1}
\Gamma^\ast_1(H)&=H^3\sum\limits_{d_1 d_2\leq \xi^\ast \atop{(d_1, d_2)=1}}\frac{\mu(d_1)\mu(d_2)
\lambda^\ast(d^2_1, d^2_2)}{d^6_1 d^6_2}\nonumber\\
&+\mathcal{O}\Bigg(\sum\limits_{d_1 d_2\leq \xi^\ast}\Big(H^2(d_1d_2)^\varepsilon
+H(d_1d_2)^{2+\varepsilon}+(d_1d_2)^{4+\varepsilon}\Big)\Bigg)\nonumber\\
&=\sigma^\ast H^3-H^3 \sum\limits_{d_1 d_2>\xi^\ast}\frac{\mu(d_1)\mu(d_2)\lambda^\ast(d^2_1, d^2_2)}{d^6_1 d^6_2}
+\mathcal{O}\Big((\xi^\ast)^{5+\varepsilon}\Big)\,,
\end{align}
where
\begin{equation}\label{sigmasumast}
\sigma^\ast=\sum\limits_{d_1, d_2=1\atop{(d_1, d_2)=1}}^\infty\frac{\mu(d_1)\mu(d_2)\lambda^\ast(d^2_1, d^2_2)}{d^6_1 d^6_2}\,.
\end{equation}
From Lemma \ref{lambdaastupperbound} we derive
\begin{equation}\label{d1d2ast>est}
\sum\limits_{d_1d_2>\xi^\ast}\frac{\mu(d_1)\mu(d_2)\lambda^\ast(d^2_1, d^2_2)}{d^6_1 d^6_2}
\ll\sum\limits_{d_1d_2>\xi^\ast}\frac{(d_1d_2)^{4+\varepsilon}}{(d_1d_2)^6}
\ll\sum\limits_{n>\xi^\ast}\frac{\tau(n)}{n^{2-\varepsilon}}
\ll(\xi^\ast)^{\varepsilon-1}\,.
\end{equation}
Using \eqref{lambdaq1q2ast}, \eqref{sigmasumast}, Lemma \ref{multiplicative2}, the Euler product and
arguing as in \cite{Dimitrov1} we find
\begin{equation}\label{productp6ast}
\sigma^\ast=\prod\limits_{p}\left(1-\frac{\lambda^\ast(p^2, 1)+\lambda^\ast(1, p^2)}{p^6}\right)\,.
\end{equation}
Now \eqref{xiHast}, \eqref{GammaHast1est1}, \eqref{d1d2ast>est} and \eqref{productp6ast} yield
\begin{equation}\label{GammaHast1est2}
\Gamma^\ast_1(H)=\prod\limits_{p}\left(1-\frac{\lambda^\ast(p^2, 1)+\lambda^\ast(1, p^2)}{p^6}\right)H^3
+\mathcal{O}\Big((\xi^\ast)^{5+\varepsilon}\Big)\,.
\end{equation}

\subsection{Estimation of $\mathbf{\Gamma^\ast_2(H)}$}

Working similar to Subsection \ref{SubsectionGammaH2est} for the sum \eqref{GammaHast2} we obtain
\begin{equation}\label{GammaHast2est1}
\Gamma^\ast_2(H)\ll H^{2+\varepsilon}(\xi^\ast)^{-\frac{1}{2}}\,.
\end{equation}

\subsection{The end of the proof of Theorem 2}

Bearing in mind \eqref{GammaHastdecomp}, \eqref{GammaHast1est2}, \eqref{GammaHast2est1}
and choosing $\xi^\ast=H^\frac{1}{2}$ we establish the asymptotic formula \eqref{asymptoticformula2}.
This completes the proof of Theorem \ref{Theorem2}.

\vskip20pt
\footnotesize
\begin{flushleft}
S. I. Dimitrov\\
\quad\\
Faculty of Applied Mathematics and Informatics\\
Technical University of Sofia \\
Blvd. St.Kliment Ohridski 8 \\
Sofia 1756, Bulgaria\\
e-mail: sdimitrov@tu-sofia.bg\\
\end{flushleft}

\begin{flushleft}
Department of Bioinformatics and Mathematical Modelling\\
Institute of Biophysics and Biomedical Engineering\\
Bulgarian Academy of Sciences\\
Acad. G. Bonchev Str. Bl. 105, Sofia-1113, Bulgaria \\
e-mail: xyzstoyan@gmail.com\\
\end{flushleft}

\end{document}